\documentclass[11pt, a4paper]{article}
\usepackage{times}
\usepackage{a4wide}
\usepackage[british]{babel}
\usepackage{enumerate, longtable}
\usepackage{amsmath, amscd, amsfonts, amsthm, amssymb, latexsym, comment, stmaryrd, graphicx}
\usepackage[T1]{fontenc}
\usepackage[latin1]{inputenc}

\newtheorem{theorem}{Theorem}

\newtheorem{definition}[theorem]{Definition}

\newtheorem{proposition}[theorem]{Proposition}

\newcommand{\NN}{\mathbb{N}}
\newcommand{\PP}{\mathbb{P}}
\newcommand{\RR}{\mathbb{R}}

\newcommand{\Real}{\mathrm{Re}}

\begin{document}
\title{A Short Note on the Bruinier-Kohnen Sign Equidistribution Conjecture and Hal\'{a}sz' Theorem}
\author{Ilker Inam\footnote{Bilecik Seyh Edebali University, Department of Mathematics, Faculty of Art and Sciences, 11020, Bilecik, Turkey, ilker.inam@gmail.com, ilker.inam@bilecik.edu.tr} and Gabor Wiese  \footnote{Universit\'e du Luxembourg,
Facult\'e des Sciences, de la Technologie et de la Communication,
6, rue Richard Coudenhove-Kalergi,
L-1359 Luxembourg, Luxembourg, gabor.wiese@uni.lu}}

\maketitle

\begin{abstract}
In this note, we improve earlier results towards the Bruinier-Kohnen sign equidistribution conjecture for half-integral weight modular eigenforms in terms of natural density by using a consequence of Hal\'{a}sz' Theorem. Moreover, applying a result of Serre we remove all unproved assumptions.\\

\textbf{Mathematics Subject Classification (2010):} 11F37 (Forms of half-integer weight; nonholomorphic modular forms); 11F30 (Fourier coefficients of automorphic forms).\\

\textbf{Keywords:} Half-integral weight modular forms, Shimura lift, Sato-Tate equidistribution, Fourier coefficients of modular forms, Hecke eigenform.
\end{abstract}

By using the celebrated Sato-Tate theorem for integral weight modular eigenforms and the Shi\-mura lift, in \cite{IlGa} and in \cite{AIW} (together with Sara Arias-de-Reyna), we prove results related to the Bruinier-Kohnen sign equidistribution conjecture for modular eigenforms of half integral weight.
In this note we improve one of our main results to a formulation in terms of natural density. Moreover, a theorem of Serre's allows us to remove all unproved assumptions.

The first improvement is due to the following application of Hal\'{a}sz' Theorem that one of us learned from Kaisa Matom{\"a}ki.

\begin{theorem}\label{Halasz}
Let $g : \NN \to \{-1,0,1\}$ be a multiplicative function.
If $\sum_{p, g(p)=0}^{} \frac{1}{p}$ converges and $\sum_{p, g(p)=-1}^{} \frac{1}{p}$ diverges then 
\begin{center}
$\lim_{x \to \infty} \frac{|\{ n \leq x : g(n) \gtrless 0    \}|}{\mid \{n \leq x : g(n) \neq 0\} \mid } = \frac{1}{2}$.
\end{center}
\end{theorem}

\begin{proof}
By Lemma~2.2 of~\cite{MR}, which is a consequence of Hal\'{a}sz' Theorem (see \cite{Hal} for details), there exists an absolute positive constant $C$ such that
\begin{center}
$\sum_{n \leq x} g(n) \leq C \cdot x \exp\left(-\frac{1}{4} \sum_{p \leq x} \frac{1-g(p)}{p} \right)$
\end{center}
for all $x \geq 2$.
By assumption, we have $1-g(p) \geq 0$ for all~$p$ and $1-g(p) =2>1$ for any $p$ with $g(p)=-1$.
We conclude that for $x \to \infty$, $\exp\left(-\frac{1}{4} \sum_{p \leq x} \frac{1-g(p)}{p} \right)$ tends to $0$. Hence for the average value of $g$, we have $\lim_{x \to \infty} \frac{\sum_{n \leq x} g(n)}{x}=0$ and therefore 
\begin{center}
$\sum_{n \leq x} g(n) = \mid \{ n \leq x | g(n) >0 \}\mid-\mid \{ n \leq x | g(n) <0 \}\mid=o(x)$.
\end{center}
Since $\sum_{p, g(p)=0} \frac{1}{p}$ converges by assumption, we conclude again by Lemma~2.2 of~\cite{MR} that for $x \to \infty$, 
\begin{center}
$\frac{\sum_{n \leq x} \mid g(n) \mid}{x}
=\frac{\mid \{ n \leq x | g(n) >0 \}\mid+\mid \{ n \leq x | g(n) <0 \}\mid}{x}$  
\end{center}
tends to a positive limit, hence the assertion follows immediately.
\end{proof}

In order to state and prove the results towards the Bruinier-Kohnen conjecture, we introduce some notation to be used throughout the note. Let $k \ge 2$ and $4|N$ be integers and $\chi$ be a quadratic Dirichlet character modulo~$N$. We denote the space of cusp forms of weight $k+1/2$ for the group $\Gamma_1(N)$
with character~$\chi$ by $S_{k+1/2}(N, \chi)$ in the sense of Shimura,
as in the main theorem in \cite{Shi} on p.~458. For Hecke operators $T_{p^2}$ for primes $p \nmid N$, let $f=\sum_{n \ge 1}^{}a(n)q^n \in S_{k+1/2}(N, \chi)$ be a non-zero cuspidal Hecke eigenform with real coefficients. For a fixed squarefree $t$ such that $a(t) \neq 0$, denote by $F_t$ the Shimura lift of~$f$ with respect to~$t$. It is a cuspidal Hecke eigenform of weight $2k$ for the group $\Gamma_0(N/2)$ with trivial character. By normalising~$f$ we can and do assume $a(t)=1$, in which case $F_t$ is normalised.

As in our previous treatments, the following theorem, of which we only state a weak version, is in the core of our approach. Its proof is based on the Sato-Tate theorem, see~\cite{ST}.

\begin{theorem}\cite{IlGa},\cite{AIW} \label{thm:prime}
Assume the set-up above and define the set of primes
$$\PP_{> 0}:=\{p : a(tp^2)  > 0 \}$$
and similarly $\PP_{< 0}$ and $\PP_{= 0}$ (depending on $f$ and~$t$).
Then the sets $\PP_{> 0}$ and $\PP_{< 0}$ have positive natural densities and the set $\PP_{= 0}$ has natural density $0$.
\end{theorem}

Due to its importance in the sequel, here we recall the following notion (Definition~2.2.1 of~\cite{AIW}).
\begin{definition}
Let $S$ be a set of primes. It is called {\em weakly regular} if there is $a \in \RR$ (called the {\em Dirichlet density} of~$S$) and a function $g(z)$ which is holomorphic on $\{ \Real(z) > 1   \}$ and continuous (in particular, finite) on $\{ \Real(z) \geq 1   \}$ such that
\begin{center}
$\sum_{p \in S} \frac{1}{p^z} = a \log \left(\frac{1}{z-1}\right)+g(z)$.
\end{center}
\end{definition}

The second improvement of this paper is the observation that a result of Serre's allows us to prove
directly that the set $\PP_{=0}$ is always weakly regular. This approach avoids the use of
Sato-Tate equidistribution and consequently does not depend on any unproved error terms for it.
It only applies to $\PP_{=0}$ and hence does not seem to give us the weak regularity of the other sets
$\PP_{>0}$ and $\PP_{<0}$.

\begin{proposition}\label{prop:weak-regularity}
Assume the setup above. Then the set $\PP_{=0}$ is weakly regular of density zero.
\end{proposition}

\begin{proof}
Let $F = F_t = \sum_{n=1}^\infty A(n) q^n$ be the Shimura lift of~$f$ with respect to~$t$.
If $F$ has CM, then the result has been proved in Theorem~4.1.1(c) of~\cite{AIW}.
So let us assume that $F$ has no CM. Due to the assumption $a(t)=1$ we have the formula
$$ A(p)=a(tp^2)+\epsilon(p)p^{k-1}$$
for all primes~$p$, where $\epsilon$ is an at most quadratic
(due to the assumption that all coefficients are real) Dirichlet character of modulus $2tN^2$
(see e.g.\ equation~(4.1) of~\cite{AIW}).
Consequently, we have the inclusion
\begin{multline*}
 \{p < x : p \nmid 2tN, p \in \PP_{=0}\}
=\{p < x : p \nmid 2tN, a(tp^2)=0\}\\
\subseteq
 \{p < x : p \nmid 2tN, A(p)=p^{k-1}\} \cup
 \{p < x : p \nmid 2tN, A(p)=-p^{k-1}\}.
\end{multline*}
By Corollaire~1 of Th\'eor\`eme~15 in~\cite{Serre} (with $h(T)=T^{k-1}$ and $h(T)=-T^{k-1}$),
it follows that
$$ \#\{p < x : p \in \PP_{=0}\} = o(\frac{x}{\log(x)^{9/8}}).$$
Consequently, Corollary~2.2.4 of~\cite{AIW} implies that $\PP_{=0}$ is
weakly regular of density zero.
\end{proof}

We now use the application of Hal\'{a}sz' theorem and the weak regularity of $\PP_{=0}$ to prove the equidistribution result we are after in terms of natural density. In \cite{AIW} we needed regularity to achieve this goal.

\begin{theorem}\label{thm:main}
Assume the setup above.
Then the sets $\{ n \in \NN | a(tn^2) >0  \}$ and $\{ n \in \NN | a(tn^2) < 0  \}$ have equal positive natural density, that is, both are precisely half of the natural density of the set $\{ n \in \NN | a(tn^2) \neq 0  \}$.
\end{theorem}

\begin{proof}
Let $g(n) = \begin{cases}
1 & \textnormal{ if } a(tn^2)>0,\\
0 & \textnormal{ if } a(tn^2)=0,\\
-1 & \textnormal{ if } a(tn^2)<0.\end{cases}$
Due to the relations $a(tn^2m^2)a(t)=a(tn^2)a(tm^2)$ for $gcd(n,m)=1$ (see p.~453 of~\cite{Shi}),
it is clear that $g(n)$ is multiplicative.
Since $\PP_{=0}$ is weakly regular of density zero by Proposition~\ref{prop:weak-regularity},
it follows that $\sum_{p\in \PP_{=0}} \frac{1}{p}$ is finite.
Moreover, the fact that $\PP_{<0}$ is of positive density implies that
$\sum_{p\in \PP_{<0}} \frac{1}{p}$ diverges.
Thus the result follows from Theorem~\ref{Halasz}.
\end{proof}

In~\cite{AIW} we obtained the same conclusion under the additional assumption of the Generalised Riemann Hypothesis (GRH) because we needed to achieve the regularity of $\PP_{=0}$, which we could derive from the very strong error term in Sato-Tate proved in~\cite{RT}, under the assumption of GRH.

\subsection*{Acknowledgements}
I.I.\ would like to thank Kaisa Matom{\"a}ki for having suggested this alternative approach to the problem. Both authors thank her for having drawn their attention again to Serre's article~\cite{Serre}.
G.W.\ acknowledges partial support by the Fonds National de la Recherche Luxembourg (INTER/DFG/12/10). The authors thank the referee for a careful reading of the article.

\end{document}